\documentclass[11pt]{amsart}
\usepackage{amsmath,amssymb,amsfonts,amsxtra}
\usepackage{verbatim}

\usepackage{times}

\usepackage{amssymb}
\usepackage{amsmath}

\makeatletter
\@namedef{subjclassname@2010}{%
  \textup{2010} Mathematics Subject Classification}
\makeatother

\newtheorem{thm}{Theorem}
\newtheorem{prop}[thm]{Proposition}

\newtheorem{lemma}[thm]{Lemma}

\theoremstyle{plain}
\numberwithin{equation}{section}

\theoremstyle{remark}

\newcommand{\Comment}[1]{}
\newcommand{\rbr}[1]{\left( {#1} \right)}

\newcommand{\cbr}[1]{\left\{ {#1} \right\}}

\newcommand{\abs}[1]{\left| {#1} \right|}
\newcommand{\norm}[1]{\left\|#1\right\|}
\newcommand{\eq}[1]{(\ref{#1})}

\usepackage{environ}
\NewEnviron{as}{\begin{align}\begin{split}\BODY\end{split}\end{align}}
\NewEnviron{as*}{\begin{align*}\begin{split}\BODY\end{split}\end{align*}}

\def\one{\mathbf{1}}

\def\RR{\mathbb{R}}

\def\ZZ{\mathbb{Z}}
\def\NN{\mathbb{N}}

\def\sinc{\text{sinc}}

\begin{document}

\title{On the sharpness of Mockenhaupt's restriction theorem}
\author{Kyle Hambrook and Izabella {\L}aba}
\date{May 30, 2013 (revised)}

\begin{abstract}

We prove that the range of exponents in Mockenhaupt's restriction theorem for Salem sets \cite{M}, with the endpoint estimate due to Bak and Seeger \cite{BS}, is optimal.

\end{abstract}

\maketitle

\section{Introduction}
Using a Stein-Tomas type argument, Mockenhaupt \cite{M} (see also Mitsis \cite{Mi}) proved the following restriction theorem, with the endpoint due to Bak and Seeger \cite{BS}.

\begin{thm}\label{Mockenhaupt}
Let $\mu$ be a compactly supported positive measure on $\RR^n$ such that for some $\alpha, \beta\in(0,n)$ we have 
\begin{equation}\label{e-ball}
\mu(B(x,r)) \leq C_1 r^{\alpha}\ \ \text{for all}\ \ x \in \mathbb{R}^n\ \text{and}\ r > 0,
\end{equation}
\begin{equation}\label{e-salem}
\abs{\widehat{\mu}(\xi)} \leq C_2 (1+|\xi|)^{-\beta/2}\ \ \text{for all}\ \ \xi \in \mathbb{R}^n.
\end{equation}
Then for all $p \geq p_{n,\alpha,\beta}:=\frac{2(2n-2\alpha+\beta)}{\beta}$, there is a $C(p) > 0$ such that 
\begin{equation}\label{estimate}
\| \widehat{fd\mu} \|_{L^{p}(\mathbb{R}^n)} \leq C(p) \| f \|_{L^2(d\mu)}
\end{equation}
for all $f \in L^2(d\mu)$. The equivalent dual form of this assertion is: For all $1 \leq p^{\prime} \leq \frac{2(2n-2\alpha+\beta)}{4(n-\alpha)+\beta}$, there is a $C(p^{\prime}) > 0$ such that
\begin{equation}\label{e-dual}
\|\widehat{f} \|_{L^2(d\mu)} \leq C(p^{\prime}) \|f\|_{L^{p^{\prime}}(\mathbb{R}^n)}
\end{equation}
for all $f \in L^{p^{\prime}}(\mathbb{R}^n)$. 
\end{thm}

When $\alpha=\beta=n-1$ and $\mu$ is the surface measure on the unit sphere $S^{n-1}$ in $\RR^n$, this is the classical Stein-Thomas theorem \cite{tomas-1}, \cite{tomas-2}, \cite{stein-beijing}, \cite{stein-ha}. The point of Theorem \ref{Mockenhaupt} is that similar estimates hold for less regular measures obeying (\ref{e-ball}) and (\ref{e-salem}), including fractal measures with $\alpha,\beta$ not necessarily integer.

It is well known (see e.g. \cite{mattila-book}, \cite{W}) that if a measure $\mu$ is supported on a set of Hausdorff dimension $\alpha_0<n$ and obeys (\ref{e-ball}) and (\ref{e-salem}), we must necessarily have $\alpha \leq \alpha_0$ and $\beta\leq\alpha_0$. The surface measure on the sphere provides an example with $\alpha=\beta=\alpha_0$. We do not know whether this is possible when $\alpha_0$ is non-integer, but there are many constructions of measures  supported on sets of fractional Hausdorff dimension $\alpha_0$ for which (\ref{e-ball}) and (\ref{e-salem}) hold with $\alpha$ and $\beta$ both arbitrarily close to $\alpha_0$.
Salem \cite{salem} constructed measures on $[0,1]$ supported on sets of Hausdorf dimension $0<\alpha<1$, and obeying (\ref{e-ball}) with the same $\alpha$, such that (\ref{e-salem}) holds for all $0<\beta<\alpha$ with the constant $C_2$ depending on $\beta$. (The verification of (\ref{e-ball}) for Salem's construction is in \cite{M}.) Further examples  are in \cite{bluhm-1}, \cite{bluhm-2}, \cite{kahane}, \cite{kaufman}, \cite{LP}.


We are interested in the question of the sharpness of the range of $p$ in Theorem \ref{Mockenhaupt}. It is easy to see that if $\mu$ is a probability measure on $\RR^n$ supported on a compact set of Hausdorff dimension $\alpha_0<n$, then (\ref{estimate}) cannot hold for any $p<2n/\alpha_0$, even if the $L^2$ norm on the right side is replaced by the stronger $L^\infty$ norm. Indeed, let $f\equiv 1$, so that $\widehat{fd\mu}=\widehat{\mu}$. The assumption on the support of $\mu$ implies that for any $\gamma>\alpha_0$ we have
$$
I_\gamma(\mu)=\int_{|\xi|\geq 1} |\widehat{\mu}(\xi)|^2 \,|\xi|^{-(n-\gamma)}d\xi=\infty
$$
(This is the usual energy integral, with the $|\xi|\leq 1$ region removed. See e.g. \cite{mattila-book}, \cite{W}.) On the other hand, by H\"older's inequality we have
$$
I_\gamma(\mu)\leq \|\widehat{\mu}\|_p^2 \,
\Big(\int_{|\xi|\geq 1}|\xi|^{-(n-\gamma)\frac{p}{p-2}}\Big)^{\frac{p-2}{p}},
$$
and the last integral is finite for $p<2n/\gamma$, so that $\|\widehat{\mu}\|_p=\infty$ for such $p$. The conclusion follows by letting $\gamma\to \alpha_0$.

In the most interesting case when $\alpha$ and $\beta$ can be taken arbitrarily close to $\alpha_0$, this leaves the intermediate range
\begin{equation}\label{e-range}
\frac{2n}{\alpha_0}\leq p <  \frac{4n - 2\alpha_0}{\alpha_0}.    
\end{equation}

In the case of the Tomas-Stein theorem, where $\mu$ is the surface measure on the unit sphere in $\RR^n$ and $\alpha=\beta=n-1$, the estimate (\ref{estimate}) is known to fail for all $p<\frac{4n-2\alpha}{\alpha}=\frac{2n+2}{n-1}$. This is seen from the so-called Knapp example, where (\ref{estimate}) is tested on characteristic functions of small spherical caps (see e.g. \cite{stein-ha}, \cite{W}). It has not been known whether similar examples exist for sets of fractional dimension. Mockenhaupt \cite{M} stated that he could not exclude the possibility that for $n=1$ and $\alpha_0=\alpha\in(0,1)$, the estimate (\ref{estimate}) could in fact hold for all $p>2/\alpha$. 
Mitsis \cite{Mi} and Bak and Seeger \cite{BS} did not try to address this question. 

In this regard, we have the following result for $n=1$.

\begin{thm}\label{main}
For $\alpha\in (0,1)$ such that $\alpha=\frac{\log( t_0)}{\log(n_0)}$ for some $t_0,n_0\in\NN$, $n_0\neq 1$, and for every $1\leq p<\frac{4}{\alpha}-2$,  the following holds. There is a probability measure $\mu$ on $[0,1]$ supported on a set $E$ of dimension $\alpha$, and a sequence of functions $\{f_\ell\}_{j\in\NN}$ on $[0,1]$ (characteristic functions of finite unions of intervals), such that
\begin{itemize}
\item $\mu$ obeys (\ref{e-ball}) with the given value of $\alpha$,
\item $\mu$ obeys (\ref{e-salem}) for every $\beta<\alpha$ (with $C_2$ depending on $\beta$),
\item the restriction estimate (\ref{estimate}) fails for the sequence $\{f_\ell\}$, i.e.
\begin{equation}\label{restriction-fail}
\frac{\| \widehat{f_\ell d\mu} \|_{L^{p}(\mathbb{R})}}{\| f_\ell \|_{L^2(d\mu)}} \rightarrow \infty
\ \ \text{as}\ \ \ell\to\infty.
\end{equation}
\end{itemize}

\end{thm}

The set of $\alpha$ in the assumptions of the theorem is dense in $(0,1)$.
It is likely that the construction could be modified to yield such a measure and sequence of functions for every $\alpha\in(0,1)$, but this would not strengthen our conclusions significantly, considering that for a fixed $p$ the relevant range of $\alpha$ is given by a strict inequality and that in any event we cannot produce a measure with $\alpha=\beta=\alpha_0$.

The Salem set $E$ will be constructed via a randomized Cantor iteration. The main idea is that, while Salem sets are random overall, they may nonetheless contain much smaller sets that come close to being arithmetically structured. In our case, $E$ will contain subsets $E\cap F_\ell$, where $F_\ell$ is a finite iteration of a smaller Cantor set with endpoints in a generalized arithmetic progression. The functions $f_\ell$ will then be characteristic functions of $F_\ell$. 

In a sense, this may be viewed as a one-dimensional analogue of Knapp's counterexample. The latter is based on the fact that an ``almost flat" spherical cap is contained in the curved sphere, or equivalently, that the sphere is tangent to a flat hyperplane. Here, the set $E$ may be thought of as random but nonetheless ``tangent" to the arithmetically structured sets $F_\ell$.

The construction of the Salem set $E$ is similar to that in \cite{LP}, but we have to be careful to make sure that the inclusion of the sets $E\cap F_\ell$ does not disturb the Fourier estimates. Our lower bound on $\| \widehat{f_\ell d\mu} \|_p$ relies on arithmetic arguments, specifically on counting solutions to linear equations in the set of endpoints of the Cantor intervals in the construction. 
Optimizing the parameters in the construction, we get Theorem \ref{main}.

If instead of Salem measures obeying (\ref{e-ball}) and (\ref{e-salem}) one considers more general measures on $\RR$ supported on sets of Hausdorff dimension $\alpha_0\in(0,1)$, then an example due to Chen \cite{chen} (based on the work of K\"orner \cite{korner}) shows that restriction estimates (\ref{estimate}) for such measures can in fact hold for all $p\geq 2/\alpha_0$. (K\"orner's measures do not necessarily obey (\ref{e-ball}) and (\ref{e-salem}) with $\alpha,\beta$ near $\alpha_0$, and it is not clear whether his construction can be modified to ensure these properties.)

It is still possible that {\it some} Salem sets do not contain structured subsets, and that the range of $p$ in (\ref{estimate}) can be improved for such sets. 
However, our result shows that Theorem \ref{Mockenhaupt} in its stated generality is optimal with regard to the range of $p$.


We also note that the same construction yields the following.

\begin{thm}\label{main-remark}
Let $\alpha$ be as in Theorem \ref{main}, and assume that the exponents $1\leq p,q<\infty$ obey
\begin{equation}\label{pq}
p<\frac{q(2-\alpha)}{\alpha(q-1)}
\end{equation}
Then there is a measure $\mu$ on $[0,1]$ and a sequence of functions $\{f_\ell\}_{\ell \in\NN}$, constructed as in the proof of Theorem \ref{main}, such that
\begin{equation}\label{restriction-fail-pq}
\frac{\| \widehat{f_\ell d\mu} \|_{L^{p}(\mathbb{R}^n)}}{\| f_\ell \|_{L^q(d\mu)}} \rightarrow \infty
\ \ \text{as}\ \ \ell\to\infty.
\end{equation}
\end{thm}


\section{The construction of $\mu$}\label{construction of measure}

Let $N_0$ and $t_0$ be integers such that $1 < t_0 < N_0$, and let $\alpha = \log t_0 / \log N_0$. 
Let also $N = N_{0}^{2n_0}$ and $t = t_0^{2n_0}$, where $n_0$ is a large integer to be chosen later. Observe that $\log t / \log N=\alpha $ regardless of the value of $n_0$, so that we may freely assume that $n_0$ is large enough while keeping $\alpha$ fixed. For short, we will write $[N]=\{0,1,\ldots,N-1\}$.

We use $C$, $C'$, etc. to denote constants that may change from line to line. Whenever such constants depend on $n_0$ or on any of the running parameters $j,k,\ell,m$,  we will indicate this explicitly by writing, e.g., $C(n_0)$; all other constants may depend on $\alpha$, but are independent of $n_0, j,k,\ell,m$.

We will construct $\mu$ and $f_\ell$ simultaneously via a sequence of Cantor iterations. 
We will have a sequence of sets $A_0, A_1, A_2, \ldots$ satisfying
\begin{as*}
A_0 &= \cbr{0}, \\
A_{j+1} &= \bigcup_{a \in A_j} (a+A_{j+1,a}), \\
A_{j+1,a} &\subset N^{-(j+1)}[N] \\
|A_{j+1,a} | &= t
\end{as*}
Note that $A_j\subset N^{-j}\mathbb{Z}$ and $|A_j |= t^j$. The freedom in the construction comes in how we choose the subsets $A_{j+1,a} \subset N^{-(j+1)}[N]$; we can make separate choices for each $j$ and each $a \in A_j$.

Given such a sequence $A_j$, we define
\begin{as}
E_j = \bigcup_{a \in A_j} a+[0,N^{-j}], \qquad
E = \bigcap_{j=1}^{\infty} E_j.
\end{as}
Since $E_1 \supset E_2 \supset \cdots$, $E$ is a closed non-empty set.

There is a natural probability measure $\mu$ on $E$, defined as the weak limit of the absolutely continuous measures $\mu_j$ with densities
\begin{as}\label{densities}
\frac{d\mu_j}{dx} = \sum_{a \in A_j} t^{-j} N^j \mathbf{1}_{[a,a+N^{-j}]}.
\end{as}

\begin{lemma}\label{Lemma 6.1}
For any choice of $A_j$ as above, $E$ has Hausdorff dimension $\alpha$, and 
$\mu$ obeys 
\begin{as*}
\mu([x,x+\epsilon]) \leq C_1(n_0) \epsilon^{\alpha} \text{ for all } \epsilon>0.
\end{as*}
\end{lemma}
\begin{proof}
This is standard. See, for example, Lemma 6.1 in \cite{LP}.
\end{proof}

We will also construct sequences of sets $P_j\subset A_j$ and $F_j\subset E_j$ so that:
\begin{itemize}
\item $P_0 = \cbr{0}$
\item $P_{j+1} = \bigcup_{a \in P_{j}} (a + N^{-(j+1)}P)$ for $j=0,1,2,\ldots$,
where $P \subset \cbr{0,1,\ldots,N-1}$ is an arithmetic progression of length $t^{1/2}=t_0^{n_0}$
\item $F_j= \bigcup_{a \in P_j} a+[0,N^{-j})$.
\end{itemize}
Note that $|P_j |= t^{j/2}$. We also define 
$$f_\ell=\one_{F_\ell}.$$

The main result of this section is the following. 

\begin{prop}\label{prop-main}
Assume that $n_0$ is sufficiently large. There is a choice of $A_j$, $j=1,2,\dots$, with the above properties such that for every $0<\beta<\alpha$ we have
\begin{equation}\label{mu-decay}
|\hat{\mu}(k)| \leq C(\beta,n_0) |k|^{-\beta/2} \qquad (k \in \mathbb{Z}\setminus\{0\}),
\end{equation}
\begin{equation}\label{flmuj-decay}
|\widehat{f_\ell \mu_j}(k)| \leq C(\beta,\ell,n_0) |k|^{-\beta/2} \qquad (k \in \mathbb{Z}\setminus\{0\}, j \geq \ell),
\end{equation}
\end{prop}

\begin{proof}
Our starting point is the construction of Salem sets in \cite{LP}, Section 6. We will modify it to make $A_j$ contain the structured sets $P_j$ while also preserving the Fourier estimates (\ref{mu-decay}), (\ref{flmuj-decay}). We will proceed by induction. Define $A_0 = \cbr{0}$, and let $A_1\subset N^{-1}[N]$ be an arbitrary set of cardinality~$t$ so that $P_1\subset A_1$. Assuming that $j\geq 1$ and that $A_j$ is given so that $P_j \subset A_j$, we define $A_{j+1}$ by constructing $A_{j+1,a}$ for each $a \in A_j$. 

 If $A \subset \mathbb{R}$ is a finite set, we will write for $k \in \mathbb{Z}$
\begin{as*}
S_A(k) = \sum_{a \in A} e^{-2\pi i a k}.
\end{as*} 

The outline is as follows. We first construct a set $B_{j+1}\subset  N^{-(j+1)}[N]$ so as to minimize the differences
\begin{equation}\label{b-e1}
\Big|\frac{1}{t}S_{B_{j+1}}(k)-\frac{1}{N}S_{N^{-(j+1)}[N]}(k)\Big|
\end{equation}
for $k\in\mathbb{Z}$, subject to the constraint that $|B_{j+1}|=t$. Moreover, we will want (\ref{b-e1}) to be similarly small if $B_{j+1}$ is replaced by any of its ``rotated" copies $B_{j+1,x}$ with $x\in[N]$ (the terminology will be explained shortly). These sets will serve as our initial candidates for $A_{j+1,a}$. Next, we choose the ``rotations" $x(a)$ for $a \in A_j$ so as to minimize the Fourier coefficients of the next generation Cantor sets with $B_{j+1,x(a)}$ used in place of $A_{j+1,a}$.

Finally, recall that we had $P_j\subset A_j$. For each $a\in P_j$, we add $N^{-(j+1)}P$ to $B_{j+1,x(a)}$, then subtract a matching number of elements of $B_{j+1,x(a)}$ that are not in $N^{-(j+1)}P$, so that the resulting set has cardinality $t$ again. This will be $A_{j+1,a}$ for $a\in P_j$. For $a\in A_j\setminus P_j$, we simply let $A_{j+1,a}=B_{j+1,x(a)}$. We will prove that these modifications can be made without destroying the Fourier estimates.

We now turn to the details. As in \cite{LP}, we will need Bernstein's inequality (see e.g. \cite{B}).

\begin{lemma}[Bernstein's inequality]\label{Lemma 6.3}
Let $X_1, \ldots, X_n$ be independent complex-valued random variables with $|X_j| \leq 1$, $\mathbb{E}X_i = 0$, and $\mathbb{E}|X_j|^2 = \sigma_j^2$. Let $\sigma > 0$ be such that $\sigma^2 \geq \sum_{j=1}^{n} \sigma_j^2$ and $\sigma^2 \geq 6 n \lambda$. Then
\begin{as*}
\mathbb{P}\rbr{\abs{\sum_{j=1}^{n} X_j} \geq n\lambda} \leq 4 \exp\rbr{-\frac{n^2\lambda^2}{8\sigma^2}}.
\end{as*}
\end{lemma}

Define $\eta_j > 0$ by
\begin{equation}\label{eta}
\eta_j^2 = 192 t^{-1}\ln(8N^{j+2}).
\end{equation}

\begin{lemma}\label{Lemma 6.2}
There is a set $B_{j+1} \subset N^{-(j+1)}[N]$ with $|B_{j+1} |= t$ such that
\begin{as}\label{e-62}
\abs{\frac{S_{B_{j+1,x}}(k)}{t} - \frac{S_{N^{-(j+1)}[N]}(k)}{N}} \leq \eta_j
\end{as}
for all $k \in \mathbb{Z}$ and $x \in \cbr{0,1,\ldots,N-1}$. Here
\begin{as*}
B_{j+1,x} = \cbr{ \frac{(x+y)\pmod{N}}{N^{j+1}} : \frac{y}{N^{j+1}} \in B_{j+1} }.
\end{as*}
\end{lemma}

\begin{proof}
This is Lemma 6.2 of \cite{LP}; we include the proof because it is short and provides a good warm-up for the main argument.

If $j$ is large enough so that $\eta_j\geq 2$, then we may choose $B_{j+1}$ to be an arbitrary subset of $N^{-(j+1)}[N]$ of cardinality $t$. Then (\ref{e-62}) holds trivially, since each term on the left side of (\ref{e-62}) is bounded by 1 in absolute value. Assume therefore that $\eta_j\leq 2$.

Let $B_{j+1} \subset N^{-(j+1)}[N]$ be a random set constructed by stipulating that for each $b\in N^{-(j+1)}[N]$ the probability that $b \in B_{j+1}$ is $p=t/N$.

Fix $k \in \mathbb{Z}$ and $x \in [N]$. For each $b \in N^{-(j+1)}[N]$, define the random variable $X_b(k,x) = (\mathbf{1}_{B_{j+1,x}}(b) - p)e^{-2 \pi i b k}$. The $X_b(k,x)$'s satisfy $\mathbb{E}_b X_b(k,x)=0$ and $\mathbb{E}_b |X_b(k,x)|^2 = p(1-p)$. Set $\sigma^2 =6 t$, $n = N$, and $\lambda = \eta_j p / 2$. Then $\sigma^2 \geq \sum_{b \in N^{-(j+1)}[N]} \mathbb{E}_b |X_b(k,x)|^2$, and $\sigma^2 \geq 6n\lambda=3\eta_j t$. 

We apply Lemma \ref{Lemma 6.3} to the $X_b(k,x)$'s.
Since
\begin{as*}
\frac{S_{B_{j+1,x}}(k)}{t} - \frac{S_{N^{-(j+1)}[N]}(k)}{N} = t^{-1}\sum_{b \in N^{-(j+1)}[N]} X_b(k,x),
\end{as*}
and
$$
4\exp\rbr{-\frac{n^2 \lambda^2}{8\sigma^2}} = 4\exp\rbr{-\ln(8N^{j+2})} = \frac{1}{2N^{j+2}},
$$
Lemma \ref{Lemma 6.3} gives
\begin{equation}\label{e-ber1}
\mathbb{P}\rbr{\abs{\frac{S_{B_{j+1,x}}(k)}{t} - \frac{S_{N^{-(j+1)}[N]}(k)}{N}} \geq \frac{\eta_j}{2}} 
=\frac{1}{2N^{j+2}}
\end{equation}
for fixed $k \in \mathbb{Z}$ and $x \in [N]$. Since $S_{B_{j+1,x}}(k)$ and $S_{N^{-(j+1)}[N]}(k)$ are periodic with period $N^{j+1}$, it suffices to consider $k \in \cbr{0,1,\ldots, N^{j+1}-1}$. Thus the probability that the event in (\ref{e-ber1}) occurs for some $k \in \mathbb{Z}$ and $x \in \cbr{0,1,\ldots,N-1}$ is bounded by $1/2$.

Hence with positive probability we have
\begin{as}\label{Lemma 6.3 proof 1}
\abs{\frac{S_{B_{j+1,x}}(k)}{t} - \frac{S_{N^{-(j+1)}[N]}(k)}{N}} \leq \frac{\eta_j}{2}
\end{as}
for all $k \in \mathbb{Z}$ and $x \in[N]$.
When $k=0$ and $x=0$, \eq{Lemma 6.3 proof 1} says $||B_{j+1} |- t| \leq \eta_j t/2$. Therefore, by either adjoining to $B_{j+1}$ or removing from it at most $\eta_j t / 2$ elements, we get a set of cardinality exactly $t$ obeying \eq{e-62} for all $k,x$ as above.
\end{proof}

The main step in the proof of Proposition \ref{prop-main} is the following lemma.

\begin{lemma}
\label{Lemma 6.4}
There is a choice of the rotations $x(a)$, $a \in A_j$, such that
\begin{as}\label{Lemma 6.4 1}
\abs{\widehat{\mu_{j+1}}(k) - \widehat{\mu_{j}}(k)} \leq C\min\rbr{1,\frac{N^{j+1}}{|k|}} t^{-(j+1)/2} \ln(8N^{j+1}).
\end{as}
for all $k \in \mathbb{Z}$, $j \geq 1$, and 
\begin{as}\label{Lemma 6.4 2}
\abs{\widehat{f_\ell\mu_{j+1}}(k) - \widehat{f_\ell \mu_{j}}(k)} \leq C\min\rbr{1,\frac{N^{j+1}}{|k|}} t^{-(j+1)/2} \ln(8N^{j+1}).
\end{as}
for all $k \in \mathbb{Z}$, $j \geq 2$, and $\ell \in \cbr{1,\ldots,j}$.
\end{lemma}


\begin{proof}
{\bf Step 1.} 
Consider the random variables
\begin{as*}
\chi_a(k) = e^{-2\pi i k a} \rbr{ \frac{S_{B_{j+1,x(a)}}(k)}{t} - \frac{S_{N^{-(j+1)}[N]}(k)}{N} },
\ \ a\in A_j,\ k\in\ZZ,
\end{as*}
where for each $a \in A_j$ we choose $x(a)$ (the same for all $k$) independently and uniformly at random from the set $[N]$. Let $c$ be a large constant.
We claim that there is a choice of $x(a)$ such that 
\begin{as}\label{Lemma 6.4 3}
\abs{ t^{-j} \sum_{a \in A_j}  \chi_a(k)  } < \lambda_j:= ct^{-(j+1)/2}\ln(8N^{j+1})
\end{as}
for all $k \in \mathbb{Z}$ and 
\begin{as}\label{Lemma 6.4 4}
\abs{ t^{-j+\ell/2} \sum_{a\in F_\ell\cap A_j} \chi_{a}(k) } < \lambda_{j,\ell}:= ct^{-\frac{j+1}{2}+\frac{\ell}{4}}\ln(8N^{j+1})
\end{as}
for all $k \in \mathbb{Z}$ and all $\ell \in \cbr{1,\ldots,j}$. 

Consider the following events:
\begin{itemize}
\item $\mathcal{E}$ is the event that $\abs{ t^{-j} \sum_{a \in A_j} \chi_{a}(k) } \geq \lambda_j$ for some $k \in \mathbb{Z}$,
\item $\mathcal{E}_\ell$ is the event that $\abs{ t^{-j+\ell/2} \sum_{a \in F_\ell \cap A_j} \chi_{a}(k) } \geq \lambda_{j,\ell}$ for some $k \in \mathbb{Z}$.
\end{itemize}
We will prove that $\mathbb{P}(\mathcal{E})<1/2$ and $\mathbb{P}(\mathcal{E}_\ell)<1/(2j)$ for $\ell=1,2,\dots,j$. Since the failure of $\mathcal{E}$ implies \eq{Lemma 6.4 3}, and the failure of all 
$\mathcal{E}_\ell$ with $\ell=1,2,\dots,j$ implies \eq{Lemma 6.4 4}, there must be a choice of $x(a)$ for which both  \eq{Lemma 6.4 3} and \eq{Lemma 6.4 4} hold.

We begin with $\mathcal{E}$. By periodicity, it suffices to consider $k\in[N^{j+1}]$.
The random variables $\chi_a(k)$, $a \in A_j$, are independent and have expectation $\mathbb{E}\chi_a(k) = 0$. By Lemma \ref{Lemma 6.2}, $|\chi_a(k)| \leq \eta_j$.
With $n=t^j$ and $\sigma^2 = cn\eta_{j}^2 = 192c t^{j-1} \ln(8N^{j+2})$, we have $\sigma^2 \geq \sum_{a \in A_j} \mathbb{E}|\chi_a(k)|^2$ and $\sigma^2 \geq 6n\lambda_j$. Therefore, by Lemma \ref{Lemma 6.3}, we have for each fixed $k$
\begin{as*}
\mathbb{P}\rbr{ \abs{ t^{-j} \sum_{a \in A_j} \chi_{a}(k) } \geq \lambda_j} 
\leq 4\exp
\left(-\frac{\lambda_j^2\, t^{2j}}{8\sigma^2}\right). 
\end{as*}
Hence $\mathcal{E}$ has probability at most $4N^{j+1}\exp
\left(-{\lambda_j^2\, t^{2j}}/{8\sigma^2}\right)$, which is less than $1/2$ if $c \geq 3072$.

Next, we turn to $\mathcal{E}_\ell$. Again, let $k\in[N^{j+1}]$. We apply Bernstein's inequality as before, but this time with $n=|F_\ell\cap A_j|=t^{\ell/2}t^{j-\ell}=t^{j-\ell/2}$ and $\sigma^2 = cn\eta_{j}^{2}
= 192 c t^{j - \ell/2 - 1} \ln(8N^{j+2})$. We get that
\begin{as*}
\mathbb{P}\rbr{ \abs{ t^{-j+\ell/2} \sum_{a \in F_\ell\cap A_j} \chi_{a}(k) } \geq \lambda_{j,\ell}} 
\leq 4\exp
\left(-\frac{\lambda_{j,\ell}^2\, t^{2j-\ell}}{8\sigma^2}\right). 
\end{as*}
Hence $\mathcal{E}_{\ell}$ has probability at most $4N^{j+1}\exp
\left(-{\lambda_{j,\ell}^2\, t^{2j-\ell}} / {8\sigma^2}\right)$, 
which is less than $1/2j$ if $c \geq 6144$.

\bigskip\noindent
{\bf Step 2. }
Define $A_{j+1}$ as follows.
Recall that $P_j\subset A_j$. For each $a\in P_j$, construct $A_{j+1,a}$ by 
adjoining $N^{-(j+1)}P$ to $B_{j+1,x(a)}$ with $x(a)$ chosen as in Step 1, then subtract a matching number of elements of $B_{j+1,x(a)}$ that are not in $N^{-(j+1)}P$, so that $N^{-(j+1)}P\subset A_{j+1,a}$ and $|A_{j+1,a}|=t$.  For $a\in A_j\setminus P_j$, we let $A_{j+1,a}=B_{j+1,x(a)}$. 
We claim that
\begin{equation}\label{e-iter1}
\left| \frac{S_{A_{j+1}}(k)}{t^{j+1}} 
-  \sum_{a\in A_j} e^{-2\pi i k a} \ \  \frac{  S_{N^{-(j+1)}[N]}  (k) }{t^j N} 
\right|
\leq 2ct^{-(j+1)/2}\ln(8N^{j+1}),
\end{equation}

\begin{equation}\label{e-iter2}
\left| \frac{  S_{A_{j+1}\cap F_\ell}(k)   }{t^{j+1}} 
- \sum_{a\in A_j\cap F_\ell} e^{-2\pi i k a} \ \   \frac{ S_{N^{-(j+1)}[N]}  (k) }{t^jN}  
\right|
\leq 2ct^{-(j+1)/2}\ln(8N^{j+1}).
\end{equation}

To see this, first let $\tilde{A}_{j+1}=\bigcup_{a\in A_j} B_{j+1,x(a)}$. Then by (\ref{Lemma 6.4 3})
$$
\left| \frac{S_{\tilde{A}_{j+1}}(k)}{t^{j+1}} 
-  \sum_{a\in A_j} e^{-2\pi i k a} \ \  \frac{  S_{N^{-(j+1)}[N]}  (k) }{t^jN} 
\right|
=\abs{ t^{-j} \sum_{a \in A_j}  \chi_a(k)  } < \lambda_j.
$$
Since $A_{j+1}$ differs from $\tilde{A}_{j+1}$ by at most $t^{(j+1)/2}$ elements, we have
$$
\left| \frac{  S_{\tilde{A}_{j+1}}(k)     }{t^{j+1}} 
-\frac{S_{{A}_{j+1}}(k)   }{t^{j+1}} 
\right| \leq t^{-(j+1)/2}.
$$
and (\ref{e-iter1}) follows. 

Similarly, by (\ref{Lemma 6.4 4})
\begin{gather*}
\left| \frac{   S_{\tilde{A}_{j+1}\cap F_\ell}(k)       }{t^{j+1}} 
-   \sum_{a\in A_j\cap F_\ell} e^{-2\pi i k a}    \ \    \frac{  S_{N^{-(j+1)}[N]}  (k)   }{t^jN} 
\right|
=\abs{ t^{-j} \sum_{a \in A_j\cap F_\ell}  \chi_a(k)  } \\
< t^{-\ell/2}\lambda_{j,\ell}
= ct^{-\frac{j+1}{2}-\frac{\ell}{4}}\ln(8N^{j+1})
\end{gather*}
Since $A_{j+1}\cap F_\ell$ differs from $\tilde{A}_{j+1}\cap F_\ell$ by at most $t^{(j+1)/2}$ elements, the left side again differs from the left side of (\ref{e-iter2}) by at most $t^{-(j+1)/2}$,
so that (\ref{e-iter2}) follows.

\bigskip

{\bf Step 3.} We will first show that \eq{e-iter1} implies \eq{Lemma 6.4 1}. We have
\begin{as*}
\widehat{\mu_j}(k) 
&= N^jt^{-j} \sum_{a \in A_j} \int_{a}^{a+N^{-j}} e^{-2\pi i k x} dx  \\
&= \frac{1-e^{-2 \pi i k / N^j}}{2\pi i k / N^j} t^{-j} S_{A_{j}}(k) \\
&= \frac{1-e^{-2 \pi i k / N^{j+1} }}{2\pi i k / N^{j+1}} 
t^{-j} \sum_{a \in A_j} e^{-2\pi i k a} \frac{S_{N^{-(j+1})[N]     }(k)}{N}, \\
\end{as*}
and
\begin{as*}
\widehat{\mu_{j+1}}(k) 
= \frac{1-e^{-2 \pi i k / N^{j+1}}}{2\pi i k / N^{j+1}} t^{-(j+1)} S_{A_{j+1}}(k) 
\end{as*}
Therefore, 
\begin{as*}
|{\widehat{\mu_{j+1}}(k) - \widehat{\mu_{j}}(k)} |
&= 
\abs{ \frac{1-e^{-2\pi i k / N^{j+1}}}{2 \pi i k / N^{j+1}} }
\left| \frac{   S_{A_{j+1}}(k)   }{t^{j+1}} 
-   \sum_{a\in A_j} e^{-2\pi i k a}  \ \   \frac{S_{N^{-(j+1)}[N]}  (k)   }{t^jN} 
\right|
\\
&\leq 2 ct^{-(j+1)/2}\ln(8N^{j+2})
\abs{ \frac{1-e^{-2\pi i k / N^{j+1}}}{2 \pi i k / N^{j+1}} }
\\
\end{as*}
Estimating the last factor by $\min(1,N^{j+1}/\pi |k|)$, 
we get \eq{Lemma 6.4 1}.

Next, we show \eq{e-iter2} implies \eq{Lemma 6.4 2}. Let $\ell \in \cbr{1,\cdots,j}$. We have
\begin{as*}
\widehat{f_\ell d\mu_j}(k) 
&= \frac{1-e^{-2\pi i k /N^{j+1}}}{2\pi i k /N^{j+1}} 
 \frac{1}{t^jN} \sum_{a\in A_j\cap F_\ell} e^{-2\pi i k a} S_{N^{-(j+1)}[N]}  (k)
\end{as*}
and
\begin{as*}
\widehat{f_{\ell} d\mu_{j+1}}(k) 
= \frac{1-e^{-2\pi i k /N^{j+1}}}{2\pi i k /N^{j+1}} t^{-(j+1)} 
S_{A_{j+1} \cap F_{\ell}}(k).
\end{as*}
Then \eq{Lemma 6.4 2} follows as above, using \eq{e-iter2} instead of \eq{e-iter1}.


\end{proof}

\begin{lemma}[cf.  \cite{LP},   Lemma 6.5]\label{Lemma 6.5}
Assume that $n_0$ is large enough. For every 
$0 < \beta < \alpha$, there is a constant $C(n_0\beta)$ such that
\begin{as*}
\sum_{j=0}^{\infty}\min\rbr{1,\frac{N^{j+1}}{|k|}}t^{-(j+1)/2}\ln(8N^{j+1}) \leq C(n_0,\beta) |k|^{-\beta/2}
\end{as*}
for all $k \in \mathbb{Z}$, $k\neq 0$.
\end{lemma}
\begin{proof}
Split the sum as 
$
\sum_{j\leq \frac{\ln |k|}{\ln N}} + \sum_{j> \frac{\ln |k|}{\ln N}}
$ 
and estimate each term separately. For details, see the proof of Lemma 6.5 of \cite{LP}.
\end{proof}

We can now conclude the proof of Proposition \ref{prop-main}.
Since $\mu_j$ converges to $\mu$ weakly, $\hat{\mu}_j$ converges to $\hat{\mu}$ pointwise. Hence
$$
|\widehat{\mu}(k)| \leq |\widehat{\mu_1}(k)| + \sum_{j=1}^{\infty} |\widehat{\mu_{j+1}}(k) - \widehat{\mu_j}(k)|.
$$
The sum is bounded by $C(n_0,\beta) |k|^{-\beta/2}$, by Lemmas \ref{Lemma 6.4} and \ref{Lemma 6.5}, 
and we have 
\begin{as*}
|\hat{\mu_1}(k)| 
= \abs{\frac{1-e^{-2\pi i k/N}}{2\pi i k/N} \frac{1}{t} \sum_{a \in A_1} e^{-2\pi i a k} }
\leq \frac{C(n_0)}{|k|}.
\end{as*}
This proves \eq{mu-decay}.

To prove \eq{flmuj-decay}, we first note the inequality
\begin{align} \label{flmuj-decay-trivial}
|\widehat{f_{\ell} d\mu_{h}}(k)|
= \abs{ \frac{1-e^{-2\pi i k /N^{h}}}{2\pi i k /N^{h}} t^{-h} 
S_{A_{h} \cap F_{\ell}}(k) }
\leq
\frac{N^h t^{-h}}{\pi |k|}|A_{h} \cap F_{\ell}|
=
\frac{N^h t^{-\ell/2}}{\pi |k|}.
\end{align}
Then \eq{flmuj-decay} is immediate in case $j = \ell$. If $j > \ell$, we write
\begin{align*}
|\widehat{f_\ell d\mu_{j}}(k)|
\leq
|\widehat{f_{\ell} d\mu_{\ell}}(k)| + |\widehat{f_\ell d\mu_{\ell+1}}(k) - \widehat{f_\ell d\mu_{\ell}}(k)|
+
\sum_{i=\ell+1}^{j-1} |\widehat{f_\ell d\mu_{i+1}}(k) - \widehat{f_\ell d\mu_{i}}(k)|.
\end{align*}
Lemmas \ref{Lemma 6.4} and \ref{Lemma 6.5} imply the sum is bounded by $C(n_0,\beta) |k|^{-\beta/2}$. For the remaining terms, we use \eq{flmuj-decay-trivial}.


\end{proof}


\section{The estimates on $f_\ell$}\label{sec3}

We start with the easy part.

\begin{lemma}\label{L2norm}
For all $1\leq q<\infty$, we have 
$\|f_\ell\|_{L^q(d\mu)}^q=\mu(F_\ell)=t^{-\ell /2}.$
\end{lemma}

\bigskip

Theorem \ref{main} will follow from this and Proposition \ref{prop-p} below.

\begin{prop}\label{prop-p}
Fix $r\in\NN$ with $r > \frac{1}{\alpha}$ and assume that $n_0$ is large enough (depending on $r$). Let $1\leq p\leq 2r$.
Then for all $\ell$ sufficiently large we have
\begin{equation}\label{est-p}
\norm{\widehat{f_\ell d\mu }}^{p}_{L^{p}(\mathbb{R})} 
\geq C(r) \frac{N^\ell r^{-\ell-1} }{ t^{\ell (p+1)/2}}.
\end{equation}
\end{prop}

\bigskip\noindent
{\it Proof of Theorems \ref{main} and \ref{main-remark}, given Proposition \ref{prop-p}.}
Fix $r$ large enough so that $r > 1/\alpha$ and $2r \geq \frac{q(2-\alpha)}{\alpha(q-1)}$.
Applying Proposition \ref{prop-p}, we see that
(\ref{est-p}) holds for all $p$ as in (\ref{pq}). Hence
\begin{as*}
\frac{\| \widehat{f_\ell d\mu} \|_{L^{p}(\mathbb{R})}}{\| f_\ell \|_{L^q(d\mu)}} 
\geq C(r)
\left(\frac{N^\ell r^{-\ell-1} }{t^{\ell (p+1)/2}}\right)^{1/p}t^{\ell/2q}.
\end{as*}
After some algebra, this is seen to go to infinity provided that (\ref{pq}) holds and that $n_0$ is large enough depending on $p$.

\bigskip

It remains to prove Proposition \ref{prop-p}. This will occupy the rest of this section, and will be done in several steps. If $Y\subset\RR$ is a finite set and $r\in\NN$, we will write
$$
M_{Y}=\# \cbr{ (a_1,\ldots,a_{2r}) \in Y^{2r}:\  \textstyle{\sum_{i=1}^{r} a_i = \sum_{i=r+1}^{2r} a_{i}} }
$$

\begin{lemma}\label{tuples}
For every $j,\ell ,r \in \mathbb{N}$ such that $j\geq \ell$,
\begin{equation}\label{e-MF}
M_{F_\ell \cap A_j}
\geq  r^{-\ell -1}  t^{(2r - 1)\ell /2} \  \left( \frac{t^{2r}}{N} \right)^{j-\ell}.
\end{equation}
\end{lemma}

\begin{proof}
Throughout the proof, the parameters $j,\ell$ will be kept fixed. Let
$$
Y=A_j\cap F_\ell,     \ \ \ |Y|=t^{\ell/2}t^{j-\ell}
$$
and 
$$ 
Z=\{a_1+\dots+a_r:\ a_1,\ldots,a_{r} \in Y\}
$$
We claim that
\begin{equation}\label{size-Z}
|Z|\leq (rt^{1/2})^\ell \, rN^{j-\ell}.
\end{equation}
Indeed, each $y\in Y$ has a unique digit representation
\begin{as*}
y= \sum_{k=1}^{\ell} y^{(k)} N^{-k}+y^{(\ell +1)}N^{-j}
\end{as*}
where $y^{(k)} \in P$ for $k=1,\ldots,\ell$ and $y^{(\ell +1)}\in [N^{j-\ell}]$.
We may assume that $P=\{x,x+d,\dots,x+(t^{1/2}-1)d\}$.
Then each $z\in Z$ can be written (not necessarily uniquely) as
\begin{as*}
z=  \sum_{k=1}^{\ell} z^{(k)} N^{-k} + z^{(\ell +1)}N^{-j}
\end{as*}
where $z^{(\ell +1)}\in \{0,1,\dots, r(N^{j-\ell}-1)\}$ and  
$$z^{(k)} \in P':= \{rx,rx+d,\dots,rx+r(t^{1/2}-1)d\}$$
for $k=1,\ldots,\ell$. Since
$|\{0,1,\dots, r(N^{j-\ell}-1)\}|\leq rN^{j-\ell}$ and $|P'|\leq rt^{1/2}$, (\ref{size-Z}) follows.

We now prove (\ref{e-MF}). For $z\in N^{-j}\ZZ$, let
$$
g(z)=\# \cbr{ (y_1,\dots,y_r) \in Y^r: \textstyle{\sum_{i=1}^{r} y_i = z} }
$$
Then $\|g\|_{\ell^1}=|Y|^r$, $\|g\|^2_{\ell^2}=M_Y$, and 
$g$ is supported on $Z$. By H\"older's inequality, $\| g\|_{\ell^1} \leq \| g \|_{\ell^2} |Z|^{1/2}$, so that
$$
M_Y \geq \frac{\| g \|_{\ell^1}^{2}}{|Z|}
\geq \frac{(t^{\ell/2}t^{j-\ell} )^{2r}}{(rt^{1/2})^\ell \, rN^{j-\ell}} 
$$
as claimed.

\end{proof}

The next lemma is Lemma 9.A.4 of \cite{W}. We will use it in the proof of Lemma \ref{fldmul}.
\begin{lemma}\label{Lemma 9.A.4}
Let $m$ be a measure on the torus $\mathbb{T} = \mathbb{R} / \mathbb{Z}$, and let $\phi$ be a Schwartz function on $\mathbb{R}$. Define a measure $m^{\prime}$ on $\mathbb{R}$ by
\begin{as*}
dm^{\prime}(x) = \phi(x)dm(\cbr{x}),
\end{as*}
where $\cbr{x}$ is the fractional part of $x$. Then for all $\xi \in \mathbb{R}$,
\begin{as*}
\widehat{m^{\prime}}(\xi) = \sum_{k \in \mathbb{Z}} \widehat{m}(k) \widehat{\phi}(\xi-k).
\end{as*}
Moreover, if there are $C > 0$ and $\alpha > 0$ such that
\begin{as*}
|\widehat{m}(k)| \leq C(1+|k|)^{-\alpha} \quad \text{for all } k \in \mathbb{Z},
\end{as*}
then there is a $C^{\prime} > 0$ such that
\begin{as*}
|\widehat{m^{\prime}}(\xi)| \leq C^{\prime}(1+|\xi|)^{-\alpha} \quad \text{for all } \xi \in \mathbb{R}.
\end{as*}
\end{lemma}

\begin{lemma}\label{fldmul}
Let $\ell,r \in \mathbb{N}$ with $r > \frac{1}{\alpha}$. Then 
\begin{equation}\label{e-up}
\norm{\widehat{f_\ell d\mu }}^{2r}_{L^{2r}(\mathbb{R})} 
\geq C_{2r} \frac{N^\ell r^{-\ell-1} }{t^{\ell (2r+1)/2}},
\end{equation}
where $$C_{2r} = \int_{\infty}^{-\infty} \rbr{\frac{\sin(\pi x)}{\pi x}}^{2r} dx \in (0,\infty).$$
\end{lemma}
\begin{proof}
By Proposition \ref{prop-main}, for every $0 < \beta < \alpha$ we have
$$
|\widehat{f_\ell d\mu_j}(k)| \leq C|k|^{-\beta/2}
$$ 
for $k \in \mathbb{Z}\setminus \cbr{0}$ and $j \geq \ell$.
By Lemma \ref{Lemma 9.A.4}, this inequality extends to 
$$
|\widehat{f_\ell d\mu_j}(\xi)| \leq C|\xi|^{-\beta/2}
$$ 
for $|\xi| \geq 1$ and $j \geq \ell$. Fix $\beta \in (0,\alpha)$ such that $r > 1/\beta > 1/\alpha$, and let $g(\xi) := \min\rbr{1,C|\xi|^{-\beta/2}}$. Assume $C > 1$ without loss of generality. We have $\abs{\widehat{f_\ell d\mu_j}} \leq g$ and $g \in L^{2r}(\mathbb{R})$. By a straightforward application of the portmanteau theorem on the weak convergence of measures (cf. \cite{B}), the fact that $\mu_j \rightarrow \mu$ weakly implies we have $\widehat{f_\ell d\mu_j} \rightarrow \widehat{f_\ell d\mu}$ pointwise. So, by the dominated convergence theorem, 
$\norm{\widehat{f_\ell d\mu_j }}_{2r} \rightarrow \norm{\widehat{f_\ell d\mu }}_{2r}$.
Therefore, it will suffice to prove that
\begin{as*}
\norm{\widehat{f_\ell d\mu_j }}^{2r}_{2r} 
\geq C_{2r} \frac{N^\ell r^{-\ell} }{t^{\ell (2r+1)/2}}
\end{as*}
for $j \geq \ell$. 

By \eq{densities} we have
$$
f_{\ell} d \mu_{j}
=  t^{-j } N^j \sum_{b \in P_\ell} \sum_{a \in A_j \cap [b,b+N^{-\ell}]} \mathbf{1}_{[a,a+N^{-j }]} dx
$$
so that
\begin{as*}
\widehat{f_\ell d\mu_j}(\xi)
&= \frac{1-e^{-2\pi i \xi /N^j }}{2\pi i \xi /N^j } t^{-j} \sum_{b \in P_\ell} \sum_{a \in A_j \cap [b,b+N^{-\ell}]} e^{-2\pi i a \xi} \\
&= e^{-\pi i \xi / N^j }\sinc(\xi / N^j )\  t^{-j } \sum_{a \in F_\ell \cap A_j } e^{-2\pi i a \xi},
\end{as*}
where $\sinc(x)=\sin(\pi x)/(\pi x)$. 
Therefore 
\begin{as*}
\norm{\widehat{f_\ell d\mu_j }}^{2r}_{2r} 
&= t^{-2rj } \int_{-\infty}^{\infty} \sinc^{2r}(\xi / N^j ) \abs{\sum_{a \in F_\ell \cap A_j } e^{-2\pi i a \xi}}^{2r} d\xi \\
&= \frac{N^j }{t^{2rj }} \int_{-\infty}^{\infty} \sinc^{2r}(\eta) \abs{\sum_{a \in N^j(F_\ell \cap A_j) } e^{-2\pi i a \eta}}^{2r} d\eta \\
&= \frac{N^j }{t^{2rj }} \int_{-\infty}^{\infty} \sinc^{2r}(\eta) \sum_{a_1,\ldots,a_2r \in N^j(F_\ell \cap A_j) } e^{-2 \pi i \eta \sum_{n=1}^{r} (a_{n} - a_{n+r})} \\
&= \frac{N^j }{t^{2rj}} \sum_{a_1,\ldots,a_2r \in N^j(F_\ell \cap A_j) } \widehat{\sinc^{2r}}\rbr{\sum_{n=1}^{r} (a_{n} - a_{n+r})}.
\end{as*}
But 
\begin{as*}
\widehat{\sinc^{2r}} = \ast_{i=1}^{2r} \widehat{\sinc} = \ast_{i=1}^{2r} \mathbf{1}_{[-1/2,1/2]}\geq 0.
\end{as*}
So
\begin{as*}
\norm{\widehat{f_\ell  d\mu_j }}^{2r}_{2r}
\geq \frac{N^j }{t^{2rj }}\ \  \widehat{\sinc^{2r}}(0) 
M_{N^j(F_\ell \cap A_j)}.
\end{as*}
Appealing to Lemma \ref{tuples} completes the proof.
\end{proof}

We can now prove Proposition \ref{prop-p}.

\begin{proof}[Proof of Proposition \ref{prop-p}]
Fix $r\in\NN$ so that $r> 1/\alpha$. By Lemma \ref{fldmul}, (\ref{est-p}) holds with $p=2r$, provided that $n_0$ is large enough. It suffices to prove that it also holds for all $p$ such that $1\leq p < 2r$.

Let $\phi$ be a function in $L^\infty(\RR)$, then for $1\leq p<2r$ we have
$$
\|\phi\|_{2r}^{2r}=\int |\phi|^{2r}=\int |\phi|^p\ |\phi|^{2r-p}\leq \|\phi\|_p^p \, \|\phi\|_\infty^{2r-p}.
$$
We apply this with $\phi=\widehat{f_\ell d\mu}$. We have $\|\widehat{f_\ell d\mu}\|_\infty
\leq \mu(F_\ell)=t^{-l/2}$, so that
$$
\|\widehat{f_\ell d\mu}\|_p^p
\geq C \frac{N^\ell r^{-\ell-1} }{t^{\ell (2r+1)/2}}\cdot(t^{\ell/2})^{2r-p}
=C \frac{N^\ell r^{-\ell-1} }{t^{\ell (p+1)/2}}
$$
as claimed.

\end{proof}

{\bf Acknowledgement.} The authors were supported in part by an NSERC Discovery Grant. We thank Andreas Seeger for pointing us to reference \cite{chen}.


\noindent{\sc Department of Mathematics, University of British Columbia, Vancouver,
B.C. V6T 1Z2, Canada}
                                                                                     
\noindent{\it hambrook@math.ubc.ca, ilaba@math.ubc.ca}


\begin{thebibliography}{MCS96}

\bibitem{BS} J.-G. Bak, A. Seeger, \textit{Extensions of the Stein-Tomas theorem}, Math. Res. Lett. \textbf{18} (2011), no. 4, 767–-781. 

\bibitem{B} P. Billingsley, \textit{Convergence of Probability Measures}, 2nd Ed., John Wiley \& Sons, Inc., New York, N.Y., 1999.


\bibitem{bluhm-1} C. Bluhm, \textit{Random recursive construction of Salem sets},
Ark. Mat. {\bf 34} (1996), 51--63.

\bibitem{bluhm-2} C. Bluhm, \textit{On a theorem of Kaufman: Cantor-type
construction of linear fractal Salem sets},
Ark. Mat. {\bf 36} (1998), 307--316.

\bibitem{chen} X. Chen, {\it A Fourier restriction theorem based on convolution powers}, preprint, 2012.


\bibitem{kahane} J.P. Kahane, \textit{Some Random Series of Functions}, Cambridge
Univ. Press, 1985.


\bibitem{kaufman} L. Kaufman, \textit{On the theorem of Jarnik and Besicovitch}, 
Acta Arith. \textbf{39} (1981), 265--267.

\bibitem{korner} T.W. K\"orner, {\it On a theorem of Saeki concerning convolution squares of singular measures}, Bull. Soc. Math. France, 136 (2008), 439--464.


\bibitem{LP} I. {\L}aba, M. Pramanik, \textit{Arithmetic progressions in sets of fractional dimension}, 
Geom. Funct. Anal. \textbf{19} (2009), no. 2, 429--456.



\bibitem{mattila-book} P. Mattila, \textit{Geometry of sets and measures in Euclidean 
spaces}, Cambridge
Studies in Advanced Mathematics, vol. 44, Cambridge University Press, 1995.


\bibitem{Mi} T. Mitsis, {\it A Stein-Tomas restriction theorem for general measures}, Publ. Math. Debrecen \textbf{60} (2002), 89--99.




\bibitem{M} G. Mockenhaupt, \textit{Salem sets and restriction properties of
Fourier transforms}, Geom. Funct. Anal. \textbf{10} (2000), 1579--1587.



\bibitem{salem}R. Salem, \textit{On singular monotonic functions whose spectrum has a 
given Hausdorff dimension}, Ark. Mat. \textbf{1} (1950), 353--365.



\bibitem{stein-beijing} E.M. Stein, \textit{Oscillatory integrals in Fourier analysis}, 
in \textit{Beijing Lectures in Harmonic Analysis} (E.M. Stein, ed.), Ann. Math. Study
\# 112, Princeton Univ. Press, 1986, pp. 307-355.

\bibitem{stein-ha} E.M.~Stein, \textit{Harmonic Analysis}, Princeton Univ. Press,
Princeton, 1993.


\bibitem{tomas-1} P.A. Tomas, \textit{A restriction theorem for the Fourier transform},
Bull. Amer. Math. Soc. \textbf{81} (1975), 477--478.

\bibitem{tomas-2} P.A. Tomas, \textit{Restriction theorems for the Fourier transform},
in \textit{Harmonic Analysis in Euclidean Spaces}, G. Weiss and S. Wainger, eds.,
Proc. Symp. Pure Math. \# 35, Amer. Math. Soc., 1979, vol, I,
pp. 111-114.




\bibitem{W} T. Wolff, \textit{Lectures on Harmonic Analysis}, 
I. {\L}aba and C. Shubin, eds., Amer. Math. Soc., Providence, R.I. (2003).

\end{thebibliography}
\end{document}